\documentclass[12pt]{amsart}
\usepackage{amsthm, amssymb}
\usepackage{amsfonts, amscd}
\usepackage{epsfig,multicol}
\usepackage{eufrak}
\theoremstyle{plain} \numberwithin{equation}{section}
\newtheorem{theo}{Theorem}[section]
\newtheorem{coro}[theo]{Corollary}
\newtheorem{prop}[theo]{Proposition}
\newtheorem{lemm}[theo]{Lemma}

\theoremstyle{definition}
\newtheorem*{defi}{Definition}
\newtheorem{exam}[theo]{Example}
\newtheorem*{rema}{Remark}

\def\Z{\mathbb Z}

\def\R{\mathbb R}

\def\bB{\bar B}

\def\G{\Gamma}
\def\EA{\mathcal E_A}
\def\EB{\mathcal E_B}
\def\tEA{\bar{\mathcal E}_A}
\def\tEB{\bar{\mathcal E}_B}

\def\bx{\overline x}

\def\f{\tilde f}
\def\MH{\mathcal H}
\def\ML{\mathcal L}

\def\MC{\mathcal A}
\def\MD{\mathcal B}
\def\MS{\mathcal S}

\def\T{\mathfrak B}
\def\bA{\bar A}
\def\vf{\varphi}

\def\tC{\tilde C}

\DeclareMathOperator{\rank}{rank}
\DeclareMathOperator{\GL}{GL}

\begin{document}
\title{Classification of real Bott manifolds}
\author[M. Masuda]{Mikiya Masuda}
\address{Department of Mathematics, Osaka City
University, Sumiyoshi-ku, Osaka 558-8585, Japan.}
\email{masuda@sci.osaka-cu.ac.jp}

\date{\today}
\thanks{The author was partially supported by Grant-in-Aid for
Scientific Research 19204007}
\subjclass[2000]{Primary 57R91; Secondary 53C25, 14M25}
\keywords{real toric manifold, real Bott tower, real Bott manifold, flat riemannian manifold.}

\begin{abstract}
A real Bott manifold is the total space of a sequence of $\R P^1$ bundles starting 
with a point, where each $\R P^1$ bundle is the projectivization of a Whitney sum 
of two real line bundles. 
A real Bott manifold is a real toric manifold which admits a flat riemannian metric. 
An upper triangular $(0,1)$ matrix with zero diagonal 
entries uniquely determines such a sequence of $\R P^1$ bundles but 
different matrices may produce diffeomorphic real Bott manifolds.  
In this paper we determine when two such matrices produce diffeomorphic 
real Bott manifolds.  The argument also proves that any graded ring isomorphism 
between the cohomology rings of real Bott manifolds with $\Z/2$ coefficients 
is induced by an affine diffeomorphism between the real Bott manifolds.  In particular, 
this implies the main theorem of \cite{ka-ma08} which asserts that  
two real Bott manifolds are diffeomorphic if and only if their cohomology rings 
with $\Z/2$ coefficients are isomorphic as graded rings. 
We also prove that the decomposition of a real Bott manifold into a product of 
indecomposable real Bott manifolds is unique up to permutations of the 
indecomposable factors. 
\end{abstract}

\maketitle 


\section{Introduction}

A {\it real Bott tower} of height $n$, which is a real analogue of a Bott tower 
introduced in \cite{gr-ka94}, is a sequence of $\R P^1$ bundles 
\begin{equation*} \label{tower}
M_n\stackrel{\R P^1}\longrightarrow M_{n-1}\stackrel{\R P^1}\longrightarrow 
\cdots\stackrel{\R P^1}\longrightarrow M_1
\stackrel{\R P^1}\longrightarrow M_0=\{\textrm{a point}\}
\end{equation*}
such that $M_j\to M_{j-1}$ for $j=1,\dots,n$ is the 
projective bundle of the Whitney sum of a real line bundle $L_{j-1}$ and the trivial 
real line bundle over $M_{j-1}$, and we call $M_n$ a {\it real Bott manifold}. 
A real Bott manifold 
naturally supports an action of an elementary abelian 2-group 
and provides an example of a real toric manifold which admits 
a flat riemannian metric invariant under the action. 
Conversely, it is shown in \cite{ka-ma08} 
that a real toric manifold which admits a flat riemannian 
metric invariant under an action of an elementary abelian 2-group 
is a real Bott manifold. 

Real line bundles are classified by their first Stiefel-Whitney classes as is 
well-known and $H^1(M_{j-1};\Z/2)$, where $\Z/2=\{0,1\}$, is isomorphic to 
$(\Z/2)^{j-1}$ through a canonical basis, so the line bundle $L_{j-1}$ is 
determined by a vector $A_j$ 
in $(\Z/2)^{j-1}$.  We regard $A_j$ as a column vector in $(\Z/2)^n$ 
by adding zero's and form an $n\times n$ matrix $A$ by putting $A_j$ as the $j$-th column. 
This gives a bijective correspondence between the set of real Bott towers of height $n$ 
and the set $\T(n)$ of $n\times n$ upper triangular $(0,1)$ matrices with zero 
diagonal entries.  Because of this reason, we may denote the real Bott manifold $M_n$ by $M(A)$.  

Although $M(A)$ is determined by the matrix $A$, it happens that 
two different matrices in $\T(n)$ produce (affinely) diffeomorphic real Bott manifolds. 
In this paper we introduce three operations on $\T(n)$ 
and say that two elements in $\T(n)$ are {\it Bott equivalent} if one is 
transformed to the other through a sequence of the three operations. 
Our first main result is the following. 

\begin{theo} \label{main}
The following are equivalent for $A,B$ in $\T(n)$: 
\begin{enumerate}
\item[(1)] $A$ and $B$ are Bott equivalent.
\item[(2)] $M(A)$ and $M(B)$ are affinely diffeomorphic.
\item[(3)] $H^*(M(A);\Z/2)$ and $H^*(M(B);\Z/2)$ are isomorphic as graded 
rings. 
\end{enumerate}
Moreover, any graded ring isomorphism from $H^*(M(A);\Z/2)$ to 
\linebreak
$H^*(M(B);\Z/2))$ is induced by an affine diffeomorphism from 
$M(B)$ to $M(A)$. 
\end{theo}

In particular, we obtain the following main theorem of \cite{ka-ma08}. 

\begin{coro}[\cite{ka-ma08}] \label{maincoro}
Two real Bott manifolds are diffeomorphic if and only if their cohomology 
rings with $\Z/2$ coefficients are isomorphic as graded rings. 
\end{coro}

It is asked in \cite{ka-ma08} whether Corollary~\ref{maincoro} holds for any 
real toric manifolds but a counterexample is given in \cite{masu08}.

We say that a real Bott manifold is \emph{indecomposable} if it is not 
diffeomorphic to a product of more than one real Bott manifolds.  
Using Corollary~\ref{maincoro} together with an idea used to prove 
Theorem~\ref{main}, we are able to prove our second main result. 

\begin{theo} \label{main1}
The decomposition of a real Bott manifold into a product of indecomposable 
real Bott manifolds is unique up to permutations of the indecomposable 
factors.  
\end{theo}

In particular, we have

\begin{coro}[Cancellation Property]  \label{main1coro}
Let $M$ and $M'$ be real Bott manifolds. If $S^1\times M$ and 
$S^1\times M'$ are diffeomorphic, then $M$ and $M'$ are diffeomorphic. 
\end{coro}

It would be interesting to ask whether Theorem~\ref{main1} and 
Corollary~\ref{main1coro} hold for any real toric manifolds.  

The author learned from 
Y. Kamishima that Corollary~\ref{main1coro} can also be obtained from 
the method developed in \cite{ka-na08} and \cite{nazr08} and that 
the cancellation property above fails to hold for general compact 
flat riemannian manifolds, see \cite{char65-1}.  

This paper is organized as follows.  In Section~\ref{sect:rbott} we describe 
$M(A)$ and its cohomology rings explicitly in terms of the matrix $A$. 
In Section~\ref{sect:matrix} we introduce the three operations on $\T(n)$. 
To each operation we associate an affine diffeomorphism between 
real Bott manifolds in Section~\ref{sect:affine}, which implies the 
implication (1) $\Rightarrow$ (2) in Theorem~\ref{main}.  The implication 
(2) $\Rightarrow$ (3)  is trivial. In Section~\ref{sect:cohom} we prove 
the latter statement in Theorem~\ref{main}. 
The argument also establishes the implication (3) $\Rightarrow$ (1).  
In the proof we introduce 
a notion of eigen-element and eigen-space in the first cohomology group 
of a real Bott manifold using the multiplicative structure of the cohomology 
ring and they play an important role on the analysis of isomorphisms between 
cohomology rings.  Using this notion, we prove 
Theorem~\ref{main1} in Section~\ref{sect:decom}.  

\section{Real Bott manifolds and their cohomology rings} \label{sect:rbott}

As mentioned in the Introduction, a real Bott manifold $M(A)$ of dimension $n$ 
is associated to a matrix $A\in\T(n)$.  
In this section we give an explicit description of $M(A)$ and 
its cohomology ring. 

We set up some notation. 
Let $S^1$ denote the unit circle consisting 
of complex numbers with unit length. For elements $z\in S^1$ and $a\in 
\Z/2$ we use the following notation 
\[
z(a):=\begin{cases} z \quad&\text{if $a=0$}\\
\bar z\quad&\text{if $a=1$}.
\end{cases}
\]
For a matrix $A$ we denote by $A^i_j$ the 
$(i,j)$ entry of $A$ and by $A^i$ (resp. $A_j$) the $i$-th row (resp. 
$j$-th column) of $A$. 

Now we take $A$ from $\T(n)$ and define 
involutions $a_i$'s on $T^n:=(S^1)^n$ by 
\begin{equation} \label{ai}
a_i(z_1,\dots,z_n):=(z_1,\dots,z_{i-1},-z_i,z_{i+1}(A^i_{i+1}),\dots,
z_n(A^i_n))
\end{equation}
for $i=1,\dots,n$. These involutions $a_i$'s commute with each other and 
generate an elementary abelian 2-group of rank $n$, denoted by $G(A)$. 
The action of $G(A)$ on $T^n$ is free and the orbit space is the desired 
real Bott manifold $M(A)$. 

$M(A)$ is a flat riemannian manifold.  
In fact, Euclidean motions $s_i$'s $(i=1,\dots,n)$ on $\R^n$ defined by 
\[
s_i(u_1,\dots,u_n):=(u_1,\dots,u_{i-1}, u_i+\frac{1}{2}, 
(-1)^{A_{i+1}^i}u_{i+1},\dots, (-1)^{A_{n}^i}u_n)
\]
generate a crystallographic group $\Gamma(A)$, where the subgroup 
generated by $s_1^2,\dots,s_n^2$ consists of all translations by $\Z^n$, 
and the action of $\G(A)$ on $\R^n$ is free and the orbit space 
$\R^n/\G(A)$ agrees with $M(A)$ through an identification $\R/\Z$ with 
$S^1$ via an exponential map $u\to \exp(2\pi\sqrt{-1}u)$. 
$M(A)$ admits an action of an elementary abelian 2-group defined by 
$(u_1,\dots,u_n)\to (\pm u_1,\dots,\pm u_n)$ and this action preserves 
the flat riemannian metric on $M(A)$.

Let $G_k$ $(k=1,\dots,n)$ be a subgroup of $G(A)$ generated by 
$a_1,\dots,a_k$. Needless to say $G_n=G(A)$. 
Let $T^k:=(S^1)^k$ be a product of first $k$-factors in 
$T^n=(S^1)^n$. Then $G_k$ acts on $T^k$ by restricting the action of 
$G_k$ on $T^n$ to $T^k$ and the orbit space $T^k/G_k$ is a real Bott 
manifold of dimension $k$.  Natural projections $T^k\to T^{k-1}$ 
for $k=1,\dots,n$ produce a real Bott tower
\[
M(A)=T^n/G_n\to T^{n-1}/G_{n-1} \to \dots\to T^1/G_1\to \text{\{a point\}}.
\]

The graded ring structure of $H^*(M(A);\Z/2)$ can be described explicitly 
in terms of the matrix $A$.  We shall recall it. 
For a homomorphism $\lambda\colon G(A)\to \Z_2=\{\pm 1\}$ we denote by 
$\R(\lambda)$
the real one-dimensional $G(A)$-module associated with $\lambda$.  
Then the orbit space of $T^n\times \R(\lambda)$ by the 
diagonal action of $G(A)$, denoted by $L(\lambda)$, defines a real line 
bundle over $M(A)$ with the first projection.  Let $\lambda_j\colon G(A)\to 
\Z_2$ $(j=1,\dots,n)$ be a homomorphism sending $a_i$ to $-1$ for $i=j$ 
and $1$ for $i\not=j$, and we set 
$$x_j=w_1(L(\lambda_j))$$
where $w_1$ denotes the first Stiefel-Whitney class. 

\begin{lemm}[see {\cite[Lemma 2.1]{ka-ma08}} for example] \label{cohoA}
As a graded ring 
$$H^*(M(A);\Z/2)=\Z/2[x_1,\dots,x_n]/(x_j^2=x_j\sum_{i=1}^nA^i_jx_i\mid 
j=1,\dots,n).$$
\end{lemm}

Let $B$ be another element of $\T(n)$. 
Since $M(A)=T^n/G(A)$ and $M(B)=T^n/G(B)$, an affine automorphism 
$\f$ of $T^n$ together with a group isomorphism 
$\phi\colon G(B)\to G(A)$ 
induces an affine diffeomorphism $f\colon M(B)\to M(A)$ 
if $\f$ is $\phi$-equivariant, 
i.e., $\f(gz)=\phi(g)\f(z)$ for $g\in G(B)$ and $z\in T^n$. Since the actions 
of $G(A)$ and $G(B)$ on $T^n$ are free, the isomorphism $\phi$ will be 
uniquely determined by $\f$ if it exists. 
We shall use $b_i$ and $y_j$ for $M(B)$ in place of $a_i$ and $x_j$ for 
$M(A)$. 

\begin{lemm} \label{f*}
If $\phi(b_i)=\prod_{j=1}^na_j^{F^i_j}$ with $F^i_j\in \Z/2$, 
then $f^*(x_j)=\sum_{i=1}^nF^i_jy_i$. 
\end{lemm}

\begin{proof} 
A map $T^n\times \R(\lambda\circ\phi)\to T^n\times \R(\lambda)$ 
sending $(z,u)$ to $(\f(z),u)$ induces a bundle map $L(\lambda\circ\phi) 
\to L(\lambda)$ covering $f\colon M(B)\to M(A)$.  
Since $(\lambda_j\circ\phi)(b_i)=F^i_j$, this implies the lemma. 
\end{proof}

\section{Three matrix operations} \label{sect:matrix}

In this section we introduce three operations on matrices 
used in later sections to analyze when 
$M(A)$ and $M(B)$ (resp. $H^*(M(A);\Z/2)$ and $H^*(M(B);\Z/2)$) 
are diffeomorphic (resp. isomorphic) for $A,B\in \T(n)$. 
In the following $A$ will denote an element of $\T(n)$. 

\medskip
\noindent
{\bf 1st operation (Op1).} For a permutation matrix $S$ of size $n$ we define 
\[
\Phi_S(A):=SAS^{-1}.
\]
To be more precise, there is a permutation $\sigma$ on a set $\{1,\dots,n\}$ 
such that $S^i_j=1$ if $i=\sigma(j)$ and $S^i_j=0$ otherwise. We note that 
if we set $B=\Phi_S(A)$, then $SA=BS$ and 
\begin{equation} \label{SA=BA}
A^i_j=(SA)^{\sigma(i)}_j=(BS)^{\sigma(i)}_j=B^{\sigma(i)}_{\sigma(j)}.
\end{equation}
$\Phi_S(A)$ may not be in $\T(n)$ but we will perform the operation 
$\Phi_S$ on $A$ only when $\Phi_S(A)$ stays in $\T(n)$. 

\medskip
\noindent
{\bf 2nd operation (Op2).} For $k\in \{1,\dots,n\}$ we define a square matrix 
$\Phi^k(A)$ of size $n$ by 
\begin{equation} \label{2nd}
\Phi^k(A)_j:=A_j+A^k_jA_k\quad\text{for $j=1,\dots,n$}.
\end{equation}
$\Phi^k(A)$ stays in $\T(n)$ and since the diagonal entries of $A$ 
are all zero and we are working over $\Z/2$, the composition 
$\Phi^k\circ \Phi^k$ is the identity; 
so $\Phi^k$ is bijective on $\T(n)$. 

\medskip
\noindent
{\bf 3rd operation (Op3).} Let $I$ be a subset of $\{1,\dots,n\}$ such that 
$A_i=A_j$ for $i,j\in I$ and $A_i\not=A_j$ for $i\in I$ and 
$j\notin I$. Since the diagonal entries of $A$ are all zero, 
the condition $A_i=A_j$ for $i,j\in I$ implies that $A^i_j=0$ for 
$i,j\in I$. 
Let $C=(C^i_k)_{i,k\in I}$ with $C^i_k\in \Z/2$ be an invertible matrix of 
size $|I|$.  Then we define a square matrix $\Phi^I_C(A)$ of size $n$ by 
\begin{equation} \label{3rd}
\Phi^I_C(A)^i_j:=\begin{cases} \sum_{k\in I}C^i_kA^k_j\quad&\text{$(i\in I)$}\\
A^i_j\quad&\text{$(i\notin I)$}.
\end{cases}
\end{equation}
$\Phi^I_C(A)$ stays in $\T(n)$ and since $C$ is invertible, $\Phi^I_C$ is 
bijective on $\T(n)$. 

\begin{defi}
We say that two elements in $\T(n)$ are {\it Bott equivalent} 
if one is transformed to the other through a sequence of the three operations 
(Op1), (Op2) and (Op3).
\end{defi}

\begin{exam}
$\T(2)$ has two elements and they are not Bott equivalent. 
$\T(3)$ has $2^3=8$ elements and they are classified into four Bott 
equivalence classes as follows: 
\begin{enumerate}
\item[(1)] The zero matrix of size $3$
\item[(2)] ${\tiny  
\begin{pmatrix} 
0 & 1 & 0\\
0 & 0 & 0\\
0 & 0 & 0\end{pmatrix}\quad  
\begin{pmatrix} 
0 & 0 & 1\\
0 & 0 & 0\\
0 & 0 & 0\end{pmatrix}\quad  
\begin{pmatrix} 
0 & 0 & 0\\
0 & 0 & 1\\
0 & 0 & 0\end{pmatrix}\quad  
\begin{pmatrix} 
0 & 0 & 1\\
0 & 0 & 1\\
0 & 0 & 0\end{pmatrix}
}$
\item[(3)] 
${\tiny 
\begin{pmatrix} 
0 & 1 & 1\\
0 & 0 & 0\\
0 & 0 & 0\end{pmatrix}
}$
\item[(4)] 
${\tiny  
\begin{pmatrix} 
0 & 1 & 0\\
0 & 0 & 1\\
0 & 0 & 0\end{pmatrix}\quad 
\begin{pmatrix} 
0 & 1 & 1\\
0 & 0 & 1\\
0 & 0 & 0\end{pmatrix}
}$
\end{enumerate}
$\T(4)$ has $2^6=64$ elements and one can check that it has twelve Bott 
equivalence classes, see \cite{ka-ma08} and \cite{nazr08}. 
Furthermore, $\T(5)$ has $2^{10}=1024$ elements and one can check that 
it has $54$ Bott equivalence classes.  
The author learned from Admi Nazra that he classified real Bott manifolds 
of dimension 5 up to diffeomorhism 
from a different viewpoint (see \cite{ka-na08}, \cite{nazr08}) and found 
the 54 Bott equivalence classes in $\T(5)$.  
The author does not know the number of Bott equivalence classes 
in $\T(n)$ for $n\ge 6$ although it is in between $2^{(n-2)(n-3)/2}$ and 
$2^{n(n-1)/2}$ (see Example~\ref{Deltan} below). 
\end{exam}

\begin{exam}
Let $\T_k(n)$ $(1\le k\le n-1)$ be a subset of $\T(n)$ such that 
$A\in \T(n)$ is in $\T_k(n)$ if and only if $A$ has exactly 
$k$ non-zero columns. There is only one Bott equivalence class in 
$\T_1(n)$ and the corresponding real Bott manifold is the product of 
a Klein bottle and $(\R P^1)^{n-2}$. 
$\T_2(3)$ has two Bott equivalence classes represented by 
\[
{\tiny  
\begin{pmatrix} 
0 & 1 & 1\\
0 & 0 & 0\\
0 & 0 & 0\end{pmatrix}}\quad \text{}\quad 
{\tiny 
\begin{pmatrix} 
0 & 1 & 0\\
0 & 0 & 1\\
0 & 0 & 0\end{pmatrix}
}
\]
But $\T_2(n)$ for $n\ge 4$ has four Bott equivalence classes;  
two of them are represented by $n\times n$ matrices with the above $3\times 3$ 
matrices at the right-low corner and $0$ in others, and the other two 
are represented by $n\times n$ matrices with the following $4\times 4$ 
matrices at the right-low corner and $0$ in others
\[
{\tiny  
\begin{pmatrix} 
0 & 0 & 0 & 1\\
0 & 0 & 1 & 0\\
0 & 0 & 0 & 1\\
0 & 0 & 0 & 0\end{pmatrix}}\quad \text{}\quad 
{\tiny 
\begin{pmatrix} 
0 & 0 & 0 & 1\\
0 & 0 & 1 & 0\\
0 & 0 & 0 & 0\\
0 & 0 & 0 & 0\end{pmatrix}
}
\]
\end{exam}

\begin{exam} \label{Deltan}
Let $\Delta(n)$ be a subset of $\T(n)$ such that 
$A\in \T(n)$ is in $\Delta(n)$ if and only if 
$A^{i}_{i+1}=1$ for $i=1,\dots,n-1$. 
Only the operation (Op2) is available on $\Delta(n)$ and 
one can change $(i,i+2)$ entry into $0$ for $i=1,\dots,n-2$ using 
the operation, so that $A$ is Bott equivalent to a matrix $\bA$ of 
this form 
{\footnotesize
\begin{equation*} \label{reduced}
\bA=\begin{pmatrix} 
0&1&0&\bA^1_4&\bA^1_5&\dots&\bA^1_{n-1}&\bA^1_n\\
0&0&1&0&\bA^2_5&\dots&\bA^2_{n-1}&\bA^2_n\\
\vdots&\vdots& \ddots &\ddots &\ddots&\ddots &\vdots & \vdots\\
0& 0& \dots&0&1 &0 & \bA^{n-4}_{n-1} & \bA^{n-4}_n\\
0& 0& \dots&0&0  &1 & 0 & \bA^{n-3}_n\\
0& 0& \dots&0&0  &0 & 1 & 0\\
0& 0& \dots&0&0  &0 & 0 & 1\\
0& 0& \dots&0&0  &0 & 0 & 0
\end{pmatrix}
\end{equation*}}
$\bA$ is uniquely determined by $A$ and two elements $A,B\in \Delta(n)$ 
are Bott equivalent if and only if $\bA=\bB$. Therefore there are exactly 
$2^{(n-2)(n-3)/2}$ Bott equivalent classes in $\Delta(n)$ 
for $n\ge 2$. 
\end{exam}

\begin{rema} 
As remarked above $\Phi_S(A)$ may not stay in $\T(n)$.  This awkwardness 
can be resolved if we consider the union of $\Phi_S(\T(n))$ over 
all permutation matrices $S$. 
The three operations above preserve the union and are bijective on it. 
This union is a natural object.  In fact, 
it is shown in \cite[Lemma 3.3]{ma-pa08} that a square matrix $A$ of size $n$ 
with entries in $\Z/2$ lies in the union if and only if all principal 
minors of $A+E$ (even 
the determinant of $A+E$ itself) are one in $\Z/2$ where $E$ denotes the 
identity matrix of size $n$. 
\end{rema}

\section{Affine diffeomorphisms} \label{sect:affine}

In this section we associate an affine diffeomorphism between real Bott 
manifolds to each operation introduced in the previous section, and 
prove the implication $(1)\Rightarrow (2)$ 
in Theorem~\ref{main}, that is 

\begin{prop}
If $A, B\in \T(n)$ are Bott equivalent, then the associated real 
Bott manifolds $M(A)$ and $M(B)$ are affinely diffeomorphic. 
\end{prop}

We set $B=\Phi_S(A), \Phi^k(A), \Phi^I_C(A)$ respectively 
for the three operations introduced in the previous section. 
In order to prove the proposition above, 
it suffices to find a group isomorphism $\phi\colon G(B)\to G(A)$ and 
a $\phi$-equivariant affine automorphism $\f$ of $T^n$ which induces 
an affine diffeomorphism from $M(B)$ to $M(A)$. 

{\it The case of the operation (Op1).} 
Let $S$ and $\sigma$ be as before. 
We define a group isomorphism $\phi_S\colon G(B)\to G(A)$ by 
\begin{equation} \label{phiS}
\phi_S(b_{\sigma(i)}):=a_{i}
\end{equation}
and an affine automorphism $\f_S$ of $T^n$ by 
\[
\f_S(z_1,\dots,z_n):=(z_{\sigma(1)},\dots,z_{\sigma(n)}).
\]
Then it follows from \eqref{ai} (applied to $b_{\sigma(i)}$) that 
the $j$-th component of $\f_S(b_{\sigma(i)}(z))$ is 
$z_{\sigma(j)}(B^{\sigma(i)}_{\sigma(j)})$ for $j\not=i$ and 
$-z_{\sigma(i)}$ for $j=i$ while that of $a_i(\f_S(z))$ is 
$z_{\sigma(j)}(A^i_j)$ for $j\not=i$ and $-z_{\sigma(i)}$ for $j=i$. 
Since $A^i_j=B^{\sigma(i)}_{\sigma(j)}$ by \eqref{SA=BA}, this shows that 
$\f_S$ is $\phi_S$-equivariant. 

It follows from Lemma~\ref{f*} and \eqref{phiS} that the affine 
diffeomorphism $f_S\colon M(B)\to M(A)$ induced from $\f_S$ satisfies 
\begin{equation} \label{op1coho}
f_S^*(x_j)=y_{\sigma(j)}\quad \text{for $j=1,\dots,n$.}
\end{equation}

{\it The case of the operation (Op2).} 
We define a group isomorphism $\phi^k\colon G(B)\to G(A)$ by 
\begin{equation} \label{phik}
\phi^k(b_i):=a_ia_k^{A^i_k}
\end{equation}
and an affine automorphism $\f^k$ of $T^n$ by 
\[
\f^k(z_1,\dots,z_n):=(z_1,\dots,z_{k-1},\sqrt{-1}z_k,z_{k+1},\dots,z_n).
\]

We shall check that $\f^k$ is $\phi^k$-equivariant, i.e., 
\begin{equation} \label{fkb}
\f^k(b_i(z))= a_ia_k^{A^i_k}(\f^k(z)).
\end{equation}
The identity is obvious when $i=k$ because $A^k_k=0$ and 
$B^k_j=A^k_j$ for any $j$ by \eqref{2nd}.  Suppose $i\not=k$.  
Then the $j$-th component of the left hand side of \eqref{fkb} is given by 
\[
\begin{cases}z_j(B^i_j)\quad&\text{for $j\not=i,k$},\\ 
-z_i \quad&\text{for $j=i$},\\
\sqrt{-1}(z_k(B^i_k)) \quad&\text{for $j=k$},
\end{cases}
\]
while that of the right hand side of \eqref{fkb} is given by 
\[
\begin{cases}z_j(A^i_j+A^k_jA^i_k)\quad&\text{for $j\not=i,k$},\\ 
-z_i(A^k_iA^i_k) \quad&\text{for $j=i$},\\
(-1)^{A^i_k}(\sqrt{-1}z_k)(A^i_k) \quad&\text{for $j=k$}.
\end{cases}
\]
Since $B^i_j=A^i_j+A^k_jA^i_k$ by \eqref{2nd}, the $j$-th components above 
agree for $j\not=i,k$.  They also agree for $j=i$ because either $A^k_i$ or 
$A^i_k$ is zero.  
We note that $B^i_k=A^i_k$ by \eqref{2nd}, and 
the $k$-th components above are both $\sqrt{-1}z_k$ when $B^i_k=A^i_k=0$ 
and $\sqrt{-1}\bar{z_k}$ when $B^i_k=A^i_k=1$. 
Thus the $j$-th components above agree for any $j$. 

Since $A^i_k=B^i_k$ for any $i$, 
it follows from Lemma~\ref{f*} and \eqref{phik} that the affine 
diffeomorphism $f^k\colon M(B)\to M(A)$ induced from $\f^k$ satisfies 
\begin{equation} \label{op2coho}
(f^k)^*(x_j)=y_j\quad\text{for $j\not=k$}, \quad \quad 
(f^k)^*(x_k)=y_k+\sum_{i=1}^nB^i_ky_i.
\end{equation}

{\it The case of the operation (Op3).} 
The homomorphism $\GL(m;\Z)\to \GL(m;\Z/2)$ induced from the 
surjective homomorphism $\Z\to \Z/2$ is known (and easily proved) 
to be surjective.  
We take a lift of the matrix $C=(C^i_k)_{i,k\in I}$ 
to $\GL(|I|,\Z)$ and denote the lift by $\tC$. Then we define 
a group isomorphism $\phi^I_C\colon G(B)\to G(A)$ by 
\begin{equation} \label{phiIC}
\phi^I_C(b_i):=\begin{cases} \prod_{k\in I}a_k^{C^i_k}\quad&
\text{for $i\in I$},\\
a_i\quad&\text{for $i\notin I$,}
\end{cases}
\end{equation}
and the $j$-th component of an affine automorphism $\f^I_{\tC}$ of $T^n$ by 
\begin{equation} \label{fIC}
\f^I_{\tC}(z)_j:=\begin{cases} 
\prod_{\ell\in I}z_\ell^{\tC^\ell_j}
\quad&\text{for $j\in I$},\\
z_j\quad&\text{for $j\notin I$.}
\end{cases}
\end{equation}

We shall check that $\f^I_{\tC}$ is $\phi^I_C$-equivariant. 
To simplify notation we abbreviate 
$\f^I_{\tC}$ and $\phi^I_C$ as $\f$ and $\phi$ respectively.  
What we prove is the identity 
\begin{equation} \label{fb=bf}
\f(b_i(z))_j=\phi(b_i)\f(z)_j.
\end{equation}
We distinguish four cases. 

{\it Case 1.} The case where $i,j\in I$. 
As remarked in the definition of (Op3), 
 $A^k_\ell=0$ whenever $k,\ell\in I$, so 
$B^i_\ell=0$ for any $\ell\in I$ by \eqref{3rd}. 
It follows from \eqref{phiIC} and \eqref{fIC} that 
\[
\f(b_i(z))_j=(-z_i)^{\tC^i_j}
\prod_{\ell\in I,\ell\not=i}z_\ell(B^i_\ell)^{\tC^\ell_j}
=(-1)^{C^i_j}\prod_{\ell\in I}z_\ell^{\tC^\ell_j}
\]
while 
\[
\phi(b_i)\f(z)_j=(\prod_{k\in I}a_k^{C^i_k})\f(z)_j
=(-1)^{C^i_j}\prod_{\ell\in I}z_\ell^{\tC^\ell_j}.
\]

{\it Case 2.} The case where $i\in I$ but $j\notin I$. 
In this case we have 
\[
\f(b_i(z))_j=z_j(B^i_j)
\]
while 
\[
\phi(b_i)\f(z)_j
=z_j(\sum_{k\in I}C^i_kA^k_j)=z_j(B^i_j)
\]
where the last identity follows from \eqref{3rd}. 

{\it Case 3.} The case where $i\notin I$ but $j\in I$. 
In this case we have 
\[
\f(b_i(z))_j=\prod_{\ell\in I}z_\ell(B^i_\ell)^{\tC^\ell_j}
\]
while 
\[
\phi(b_i)\f(z)_j
=(\prod_{\ell\in I}z_\ell^{\tC^\ell_j})(A^i_j)
=\prod_{\ell\in I}z_\ell(A^i_j)^{\tC^\ell_j}.
\]
Since $B^i_\ell=A^i_\ell$ for $i\notin I$ by \eqref{3rd}, 
the above verifies \eqref{fb=bf}. 

{\it Case 4.} The case where $i,j\notin I$. 
In this case 
\[
\f(b_i(z))_j=z_j(B^i_j)
\]
while 
\[
\phi(b_i)\f(z)_j=z_j(A^i_j). 
\]
Since $B^i_j=A^i_j$ for $i\notin I$ by \eqref{3rd}, the above 
verifies \eqref{fb=bf}. 

It follows from Lemma~\ref{f*} and \eqref{phiIC} that the affine 
diffeomorphism $f^I_C\colon M(B)\to M(A)$ induced from $\f^I_C$ satisfies 
\begin{equation} \label{op3coho}
(f^I_C)^*(x_j)=\begin{cases}\sum_{i\in I}C^i_jy_i \quad&\text{for $j\in I$,}\\
y_j\quad&\text{for $j\notin I$.}
\end{cases}
\end{equation}

\section{Cohomology isomorphisms} \label{sect:cohom}

In this section we prove the latter statement in Theorem~\ref{main} 
and the implication (3) $\Rightarrow$ (1) at the same time, i.e. the 
purpose of this section is to prove the following. 

\begin{prop} \label{MAMBcoho}
Any isomorphism $H^*(M(A);\Z/2)\to H^*(M(B);\Z/2)$ is induced from 
a composition of affine diffeomorphisms 
corresponding to the three operations (Op1), (Op2) and (Op3), 
and if $H^*(M(A);\Z/2)$ and $H^*(M(B);\Z/2)$ are isomorphic as graded 
rings, then $A$ and $B$ are Bott equivalent. 
\end{prop}

We introduce a notion and prepare a lemma. 
Remember that 
\begin{equation} \label{HMA}
H^*(M(A);\Z/2)=\Z/2[x_1,\dots,x_n]/(x_j^2=x_j\sum_{i=1}^nA^i_jx_i\mid 
j=1,\dots,n).
\end{equation}
One easily sees that 
products $x_{i_1}\dots x_{i_q}$ $(1\le i_1<\dots<i_q\le n)$ form a 
basis of $H^q(M(A);\Z/2)$ as a vector space over $\Z/2$ so that 
the dimension of $H^q(M(A);\Z/2)$ is $\binom{n}{q}$ 
(see \cite[Lemma 5.3]{ma-pa08}). 

We set 
\begin{equation*} \label{alphaj}
\alpha_j=\sum_{i=1}^nA^i_jx_i\quad\text{for $j=1,\dots,n$}
\end{equation*}
where $\alpha_1=0$ since $A$ is an upper triangular matrix 
with zero diagonal entries. 
Then the relations in \eqref{HMA} are written as 
\begin{equation} \label{alpha}
x_j^2=\alpha_jx_j \quad\text{for $j=1,\dots,n$.}
\end{equation}
Motivated by this identity we introduce the following notion. 

\begin{defi} 
We call an element $\alpha\in H^1(M(A);\Z/2)$ 
an {\it eigen-element} of $H^*(M(A);\Z/2)$ 
if there exists $x\in H^1(M(A);\Z/2)$ such that $x^2=\alpha x$, $x\not=0$ 
and $x\not=\alpha$. 
The set of all elements $x\in H^1(M(A);\Z/2)$ satisfying the equation 
$x^2=\alpha x$ is a vector subspace of $H^1(M(A);\Z/2)$ which 
we call the {\it eigen-space} of 
$\alpha$ and denote by $\EA(\alpha)$. We also introduce a notation 
$\tEA(\alpha)$ which is the quotient of $\EA(\alpha)$ by the subspace 
spanned by $\alpha$, and call it the {\it reduced eigen-space} of 
$\alpha$. 
\end{defi} 

Eigen-elements and (reduced) eigen-spaces are invariants 
preserved under graded ring isomorphisms. 
By \eqref{alpha} $\alpha_j$'s are eigen-elements of 
$H^*(M(A);\Z/2)$ and the following lemma shows that 
these are the only eigen-elements. 

\begin{lemm} \label{EAa}
If $\alpha$ is an eigen-element of $H^*(M(A);\Z/2)$, then $\alpha=\alpha_j$ 
for some $j$ and the eigen-space $\EA(\alpha)$ of $\alpha$ is 
generated by $\alpha$ and $x_i$'s with $\alpha_i=\alpha$. 
\end{lemm}

\begin{proof}
By the definition of eigen-element there exists 
a non-zero element $x\in H^1(M(A);\Z/2)$ different from $\alpha$ 
such that $x^2=\alpha x$. Since both $x$ and $x+\alpha$ are non-zero, 
there exist $i$ and $j$ such that $x=x_i+p_i$ and $x+\alpha=x_j+q_j$ 
where $p_i$ is a 
linear combination of $x_1,\dots,x_{i-1}$ and $q_j$ is a linear 
combination of $x_1,\dots,x_{j-1}$. Then   
\[
x_ix_j+x_iq_j+x_jp_i+p_iq_j=0
\]
because $x(x+\alpha)=0$.
As remarked above, products $x_{i_1}x_{i_2}$ $(1\le i_1<i_2\le n)$ form 
a basis of $H^2(M(A);\Z/2)$, so $i$ must be equal to $j$ for the identity 
above to hold. 
Then as $x_j^2=x_j\alpha_j$, it follows from the identity above that 
$\alpha_j=q_j+p_i$ (and $p_iq_j=0$). This implies that $\alpha=\alpha_j$, 
proving the former statement of the lemma. 

We express a non-zero element $x\in \EA(\alpha)$ 
as $\sum_{i=1}^nc_ix_i$ $(c_i\in \Z/2)$ and let $m$ be 
the maximum number among $i$'s with $c_i\not=0$. 

\emph{Case 1.}  The case where $x_m$ appears when we express $\alpha$ 
as a linear combination of $x_1,\dots,x_n$. We express $x(x+\alpha)$ 
as a linear combination of the basis elements $x_{i_1}x_{i_2}$ 
$(1\le i_1<i_2\le n)$.  Since $x_m$ appears in both $x$ and $\alpha$, it 
does not appear in $x+\alpha$.  Therefore the term in $x(x+\alpha)$ which 
contains $x_m$ is $x_m(x+\alpha)$ and it must vanish because 
$x(x+\alpha)=0$.  Therefore $x=\alpha$. 

\emph{Case 2.} The case where $x_m$ does not appear in the linear 
expression of $\alpha$. In this case, 
the term in $x(x+\alpha)$ which contains $x_m$ is 
$x_m(x_m+\alpha)=x_m(\alpha_m+\alpha)$ since $x_m^2=\alpha_m x_m$, 
and it must vanish because 
$x(x+\alpha)=0$.  Therefore $\alpha_m=\alpha$. 
The sum $x+x_m$ is again an element of $\EA(\alpha)$. 
If $x\not=x_m$ (equivalently $x+x_m$ is non-zero), then the same argument 
applied to $x+x_m$ shows that there exists $m_1(\not=m)$ such that 
$\alpha_{m_1}=\alpha$ and $x+x_m+x_{m_1}$ is again an element of 
$\EA(\alpha)$. Repeating 
this argument, $x$ ends up with a linear combination of 
$x_i$'s with $\alpha_i=\alpha$. 
\end{proof}

With this preparation we shall prove Proposition~\ref{MAMBcoho}. 

\begin{proof}[Proof of Proposition~\ref{MAMBcoho}] 
Let $B$ be another element of $\T(n)$.  
We denote the canonical basis of $H^*(M(B);\Z/2)$ by $y_1,\dots,y_n$ 
and the elements in $H^1(M(B);\Z/2)$ corresponding to $\alpha_j$'s 
by $\beta_j$'s, i.e., 
$\beta_j=\sum_{i=1}B^i_jy_i$ for $j=1,\dots,n$.

Let $\varphi\colon H^*(M(A);\Z/2)\to H^*(M(B);\Z/2)$ be a graded ring 
isomorphism.  
It preserves the eigen-elements and (reduced) eigen-spaces. 
In the following we shall show that we can change $\varphi$ into the 
identity map by composing isomorphisms induced from 
affine diffeomorphisms corresponding to 
the three operations (Op1), (Op2) and (Op3). 

Through the operation (Op1) we may assume that 
$\varphi(\alpha_j)=\beta_j$ for any $j$ because of \eqref{op1coho}.  
Then $\varphi$ restricts to an isomorphism $\EA(\alpha_j) 
\to \EB(\beta_j)$ between eigen-spaces and induces an isomorphism 
$\tEA(\alpha_j)\to \tEB(\beta_j)$ between reduced eigen-spaces. 

Let $\alpha$ (resp. $\beta$) stand for $\alpha_j$ (resp. $\beta_j$) 
and suppose that $\varphi(\alpha)=\beta$.  Let $I$ be a subset of 
$\{1,\dots,n\}$ such that $\alpha_i=\alpha$ if and only if $i\in I$. 
We denote the image of $x_i$ (resp. $y_i$) in $\tEA(\alpha)$ (resp. 
$\tEB(\beta)$) by 
$\bar x_i$ (resp. $\bar y_i$).  The $\bar x_i$'s (resp. $\bar y_i$'s) 
for $i\in I$ form a basis of $\tEA(\alpha)$ (resp. $\tEB(\beta)$) 
by Lemma~\ref{EAa}, so if we express $\varphi(\bar x_j)=
\sum_{i\in I}C^i_j\bar y_i$ with $C^i_j\in \Z/2$, then the matrix 
$C=(C^i_j)_{i,j\in I}$ is invertible.  Therefore, through the  
operation (Op3), we may assume that $C$ is the identity matrix because of 
\eqref{op3coho}. This means that we may assume that 
$\varphi(x_j)=y_j$ or $y_j+\beta_j$ for each $j=1,\dots,n$. 
Finally through the operation (Op2), we may assume that 
$\varphi(x_j)=y_j$ for any $j$ because of \eqref{op2coho} and hence 
$A=B$ (and $\varphi$ is the identity) because $\varphi(\alpha_j)
=\beta_j$, $\alpha_j=\sum_{i=1}^nA^i_jx_i$ and $\beta_j=\sum_{i=1}^nB^i_jy_i$ 
for any $j$, proving the proposition. 
\end{proof}

\section{Unique decomposition of real Bott manifolds} \label{sect:decom}

We say that a real Bott manifold is \emph{indecomposable} if it is not 
diffeomorphic to a product of more than one real Bott manifolds.  
The purpose of this section is to prove Theorem~\ref{main1} in the 
Introduction, that is 

\begin{theo} \label{bdeco}
The decomposition of a real Bott manifold into a product of indecomposable 
real Bott manifolds is unique up to permutations of the indecomposable 
factors.  Namely, if $\prod_{i=1}^k M_i$ is 
diffeomorphic to $\prod_{j=1}^\ell N_j$ where $M_i$ and $N_j$ are 
indecomposable real Bott manifolds, then $k=\ell$ and there is a permutation 
$\sigma$ on $\{1,\dots,k=\ell\}$ such that $M_i$ is diffeomorphic to 
$N_{\sigma(i)}$ for $i=1,\dots,k$. 
\end{theo}

$H^*(\prod_{i=1}^kM_i;\Z/2)=\bigotimes_{i=1}^kH^*(M_i;\Z/2)$
by K\"unneth formula  
and the diffeomorphism types of real Bott manifolds are 
detected by cohomology rings with $\Z/2$ coefficient by 
Corollary~\ref{maincoro}, so 
the theorem above reduces to a problem on the decomposition of 
a cohomology ring into tensor products over $\Z/2$. 

We call a graded ring over $\Z/2$ 
a {\it Bott ring} of rank $n$ if it is isomorphic to the cohomology ring 
with $\Z/2$ coefficient of a real Bott manifold of dimension $n$. 
Let $\MH$ be a Bott ring of rank $n$, so it has an expression 
\begin{equation} \label{MH}
\MH=\Z/2[x_1,\dots,x_n]/(x_j^2=x_j\sum_{i=1}^nA^i_jx_i\mid 
j=1,\dots,n)
\end{equation} 
with $A\in\T(n)$.  
The eigen-elements of $\MH$ are 
\begin{equation} \label{eigenj}
\text{$\alpha_j=\sum_{i=1}^nA^i_jx_i$\quad $(j=1,\dots,n)$.}
\end{equation}

We denote by $\MH^q$ the degree $q$ part of $\MH$ and define 
\[
\begin{split}
N(\MH)&:=\{ x\in \MH^1\mid x^2=0\},\text{ and}\\
S(\MH)&:=\{x\in \MH^1\backslash\{0\}\mid \exists \bx\in\MH^1\backslash\{0\}
\ \text{with $x\bx=0$ and $\bx\not=x$}\}.
\end{split}
\]
In terms of eigen-elements and eigen-spaces, 
$N(\MH)$ is the eigen-space of the zero eigen-element. 
Also, if we write $\bx=x+\alpha$ with $\alpha\in\MH^1$, then 
$x\bx=0$ means that $x^2=\alpha x$; so $S(\MH)$ with 
the zero element added is the union of eigen-spaces of all non-zero 
eigen-elements in $\MH$. 
The latter statement in Lemma~\ref{EAa} shows that the eigen-element 
$\alpha$ is uniquely determined by $x$, hence so is $\bx$.

$N(\MH)=\MH^1$ if and only if 
$A$ in \eqref{MH} is the zero matrix. 
Unless $N(\MH)=\MH^1$, $S(\MH)\not=\emptyset$. 

\begin{lemm} \label{MHS}
The graded subring $\MH_S$ of a Bott ring $\MH$ generated by $S(\MH)$ 
is a Bott ring.
\end{lemm}

\begin{proof}
The isomorphism class of $\MH$ does not change through the three operations 
(Op1), (Op2) and (Op3).  Through (Op1) we may assume that the first 
$\ell$ columns of the matrix $A$ in \eqref{MH} are all zero but none 
of the remaining columns is zero.  If the maximum number of 
linearly independent vectors in the first $\ell$ rows of $A$ is $m$, then 
we may assume that the first $\ell-m$ rows are zero by applying the operation 
(Op3) to the first $\ell$ columns. Then $\MH_S$ is the Bott ring 
associated with the $(n-\ell+m)\times(n-\ell+m)$ submatrix of $A$ at the 
right-low corner of $A$. 
\end{proof}

Let $\MH_S$ be as in Lemma~\ref{MHS} and 
let $V$ be a subspace of $N(\MH)$ complementary to $N(\MH)\cap \MH_S^1$. 
The dimension of $V$ is $\ell-m$ in the proof of Lemma~\ref{MHS}.  
The graded subalgebra of $\MH$ generated by $V$ is an exterior algebra 
$\Lambda(V)$, so 
\begin{equation} \label{MH0}
\MH=\Lambda(V)\otimes \MH_S.
\end{equation}

We say that a Bott ring $\MH$ is \emph{semisimple} if $\MH$ is generated by 
$S(\MH)$. 
Clearly $\MH_S$ is semisimple and $\MH$ is semisimple if and only if 
$\MH=\MH_S$. 

\begin{lemm} \label{YMA}
Let $\MH$ be a Bott ring. 
If $\MH=\bigotimes_{i=1}^r\MH_i$ with Bott subrings $\MH_i$'s 
of $\MH$, then $S(\MH)=\coprod_{i=1}^rS(\MH_i)$. 
Therefore $\MH$ is semisimple if and only if all $\MH_i$'s are semisimple. 
\end{lemm}

\begin{proof}
Let $x\in S(\MH)$ and write 
$x=\sum_{i=1}^r y_i$ and $\bx= \sum_{i=1}^r z_i$ with $y_i,z_i\in \MH_i$. 
Since $x\bx=0$, we have 
\[
\text{$y_iz_j+y_jz_i=0$ for all $i\not=j$. }
\]
Suppose that $y_i\not=0$ and $z_j\not=0$ for some $i\not=j$. 
Then $y_i=z_i$ and $y_j=z_j$ to satisfy the equations above.  
This shows that 
$x=\bx$, which contradicts the fact that 
$x\in S(\MH)$. Therefore $x=y_i$ and $\bx=z_i$ for some $i$, proving the 
lemma. 
\end{proof}

Recall that a Bott ring $\MH$ has a decomposition 
$\Lambda(V)\otimes \MH_S$ in \eqref{MH0}.  

\begin{coro} \label{semis}
If $\MH$ has another 
decomposition $\Lambda(U)\otimes \MS$ where 
$U$ is a subspace of $N(\MH)$ and $\MS$ is a semisimple subring of $\MH$, 
then $\dim U=\dim V$ and $\MS=\MH_S$.
\end{coro}

\begin{proof}
Since both $S(\Lambda(U))$ abd $S(\Lambda(V))$ are empty, 
$S(\MH)=S(\MS)$ by Lemma~\ref{YMA} and this implies the corollary.
\end{proof}

\begin{lemm}  \label{factor}
Let $\MH=\bigotimes_{i=1}^r\MH_i$ be as in Lemma~\ref{YMA} and $\pi_i\colon 
\MH\to \MH_i$ be the projection. Let $\ML$ be a semisimple Bott ring and let 
$\psi\colon \ML\to \MH$ be a graded ring monomorphism.  If the composition 
$\pi_i\circ \psi\colon \ML\to \MH_i$ is an isomorphism for some $i$, then 
$\psi(\ML)=\MH_i$.
\end{lemm}

\begin{proof}
Let $y\in S(\ML)$. Then $\psi(y)\in S(\MH)$ because $\psi$ is a graded ring 
monomorphism, and it is actually in $S(\MH_i)$ by Lemma~\ref{YMA} since 
$(\pi_i\circ \psi)(y)\not=0$. This shows that $\psi(S(\ML))\subset S(\MH_i)$ 
but since $\pi_i\circ\psi$ is an isomorphism, the inclusion should be the equality. 
Therefore $\psi(\ML)=\MH_i$ because $\ML$ and $\MH_i$ are both 
semisimple. 
\end{proof} 

We say that a semisimple Bott ring is \emph{simple} if it is not 
isomorphic to the tensor product (over $\Z/2)$ of more than one semisimple 
Bott rings, 
in other words, a simple Bott ring is a Bott ring 
isomorphic to the cohomology ring 
(with $\Z/2$ coefficient) of an indecomposable real Bott manifold different 
from $S^1$.  A Bott ring isomorphic to the cohomology ring of the Klein 
bottle with $\Z/2$ coefficient is simple and we call it especially a 
\emph{Klein ring}.  
If an element $x\in S(\MH)$ satisfies $(x+\bx)^2=0$, then 
the subring generated by $x$ and $\bx$ is a Klein ring and we call 
such a pair $\{x,\bx\}$ a \emph{Klein pair}.  We note that $x$ and $\bx$ have 
the same eigen-element and $\{x,\bx\}$ is 
a Klein pair if and only if the eigen-element of $x$ and $\bx$, that is 
$x+\bx$, lies in $N(\MH)$. 

\begin{lemm} \label{klein}
If $S(\MH)\not=\emptyset$, then a Klein pair exists in $\MH$ and 
the quotient of $\MH$ by the ideal generated by a Klein pair 
is again a Bott ring. 
\end{lemm}

\begin{proof}

Let $\MH$ be of the form \eqref{MH}.  
The assumption $S(\MH)\not=\emptyset$ is equivalent to $A$ being non-zero 
as remarked before.  As in the proof of Lemm~\ref{MHS}, 
we may assume through the operation (Op1) that the first $\ell$ columns 
of $A$ are zero and none of the remaining columns is zero. Then $x_1,\dots,x_{\ell}$ 
are elements of $N(\MH)$ and the eigen-element $\alpha_{\ell+1}$ 
of $x_{\ell+1}$ is a linear combination of $x_1,\dots,x_{\ell}$, so $\alpha_{\ell+1}$ 
lies in $N(\MH)$ which means that $\{x_{\ell+1},\bx_{\ell+1}\}$ is a Klein pair. 

If $\{x,\bx\}$ is a Klein pair, then the eigen-element of $x$ is non-zero and belongs 
to $N(\MH)$, so through the operation (Op1) we may assume that it is 
$\alpha_{\ell+1}$.  Then, applying the operation (Op3) to the eigen-space of 
$\alpha_{\ell+1}$, we may assume $x=x_{\ell+1}$.  
We further may assume $\alpha_{\ell+1}=x_\ell$ by applying the operation (Op3) 
to $N(\MH)$.  The quotient 
ring of $\MH$ by the ideal generated by the Klein pair $\{x,\bx\}$ is then nothing 
but to take $x_\ell=x_{\ell+1}=0$ in $\MH$, so it is a Bott 
ring associated with a $(n-2)\times(n-2)$ matrix obtained from $A$ by 
deleting $\ell$-th and $\ell+1$-st columns and rows.  
\end{proof}

Now we are in a position to prove the unique decomposition of a semisimple 
Bott ring into a tensor product of simple Bott rings.

\begin{prop} \label{simpl}
Let ${\MC}_i$ $(i=1,\dots,p)$ and ${\MD}_j$ $(j=1,\dots,q)$ be 
simple Bott rings. If there exists a graded ring isomorphism 
\begin{equation} \label{MA}
\vf\colon \bigotimes_{i=1}^p\MC_i\to \bigotimes_{j=1}^q \MD_j,
\end{equation}
then $p=q$ and $\vf$ preserves the factors, i.e. 
there is a permutation $\rho$ on 
$\{1,\dots,p=q\}$ such that $\vf(\MC_i)=\MD_{\rho(i)}$ for $i=1,\dots,p$. 
\end{prop}

\begin{proof}
We set $\MC=\bigotimes_{i=1}^p\MC_i$ and $\MD=\bigotimes_{j=1}^q\MD_j$. 
If either $\MC$ or $\MD$ is simple (i.e. $p=1$ or $q=1$), 
then both of them must be simple and the proposition is trivial. 
In the sequel we will assume that both $\MC$ and $\MD$ are not simple 
(so that $p\ge 2$ and $q\ge 2$), and 
prove the proposition by induction on the rank of $\MC$, that is, 
$\dim \MC^1$. 

If $\vf(\MC_i)=\MD_j$ for some $i$ and $j$, say $\vf(\MC_p)=\MD_q$, 
then we factorize them so that $\vf$ induces an isomorphism 
$\bar \vf\colon \bigotimes_{i=1}^{p-1}\MC_i\to \bigotimes_{j=1}^{q-1} \MD_j$.  
By the induction assumption, we conclude $p=q$ and may assume that 
$\bar \vf(\MC_i)=\MD_i$ for $i=1,\dots,p-1$ if necessary by permuting the suffixes 
of $\MD_j$'s. Then it follows from Lemma~\ref{factor} that 
$\vf(\MC_i)=\MD_{i}$ for $i=1,\dots,p-1$. 
This together with $\vf(\MC_p)=\MD_q$ where $p=q$ proves the statement in the lemma. 
In the sequel, it suffices 
to show that $\vf(\MC_i)=\MD_j$ for some $i$ and $j$ when we have 
an isomorphism $\vf$ in the proposition.

\emph{Case 1.}  The case where some $\MC_i$ or $\MD_j$ is a Klein ring. 
We may assume that $\MC_p$ is a Klein ring without loss of generality. 
Let $\{x,\bx\}$ be a Klein pair in $\MC_p$. Its image by $\vf$ 
sits in some $\MD_j$ by Lemma~\ref{YMA} and 
we may assume that it sits in $\MD_q$. If $\MD_q$ is also a Klein ring, 
then $\vf(\MC_p)=\MD_q$. 
Therefore we may assume that $\MD_q$ is not a Klein ring in the following. 

Our isomorphism $\vf$ induces an isomorphism 
\[
\bar \vf \colon \MC/(x,\bx)=\bigotimes_{i=1}^{p-1}\MC_i\cong 
\bigotimes_{j=1}^{q-1}\MD_j\otimes (\MD_q/(\vf(x),\vf(\bx)))
\]
where $(u,v)$ denotes the ideal generated by the elements $u$ and $v$ and 
$\MD_q/(\vf(x),\vf(\bx))$ is a Bott ring by Lemma~\ref{klein}. 
Since $\rank (\MC/(x,\bx))=\rank \MC-2$, it follows 
from the induction assumption that $p-1\ge q$ and we may assume that 
$\bar \vf(\MC_i)=\MD_i$ for $i=1,\dots,q-1$ and 
$\bar \vf(\otimes_{i=q}^{p-1}\MC_i)=\MD_q/(\vf(x),\vf(\bx))$, in particular, 
$\bar \vf(\MC_1)=\MD_1$ as $q\ge 2$.  
Then, it follows from Lemma~\ref{factor} that $\vf(\MC_1)=\MD_1$. 

\emph{Case 2.} The case where none of $\MC_i$'s and $\MD_j$'s is 
a Klein ring.  Let $\{x,\bx\}$ be a Klein pair of $\MC_p$ and we may 
assume that its image by $\vf$ sits in $\MD_q$ as before.  Then $\vf$ 
induces an isomorphism 
\[
\bar \vf\colon \MC/(x,\bx)=\bigotimes_{i=1}^{p-1}\MC_i\otimes (\MC_p/(x,\bx))
\to \bigotimes_{j=1}^{q-1}\MD_j\otimes (\MD_q/(\vf(x),\vf(\bx))),
\]
where the quotients $\MC_p/(x,\bx)$ and $\MD_q/(\vf(x),\vf(\bx))$ are both 
Bott rings by Lemma~\ref{klein}.  The induction assumption can be applied 
to this situation as before. 
If $\bar \vf(\MC_i)=\MD_j$ for some $1\le i\le p-1$ and $1\le j\le q-1$, then 
$\vf(\MC_i)=\MD_j$ by Lemma~\ref{factor}. 
If $\bar \vf(\bigotimes_{i=1}^{p-1}\MC_i)=\MD_q/(\vf(x),\vf(\bx))$ 
and $\bar \vf(\MC_p/(x,\bx))=\bigotimes_{j=1}^{q-1}\MD_j$,  then 
$\vf$ restricts to an isomorphism 
\[
(\bigotimes_{i=1}^{p-1}\MC_i)\otimes\langle x,\bx\rangle\to \MD_q
\]
where $\langle x,\bx\rangle$ denotes 
the Klein ring generated by $x$ and 
$\bx$, and this contradicts the fact that $\MD_q$ is simple as $p\ge 2$.  
\end{proof}

Now Theorem~\ref{bdeco} follows from Corollaries~\ref{maincoro}, 
\ref{semis} and Proposition~\ref{simpl}.

\bigskip

\noindent
{\bf Acknowledgment}. I would like to thank Y. Kamishima for 
communications which stimulated this research and A. Nazra for 
informing me of his classification of real Bott manifolds of dimension 
$\le 5$ up to diffeomorphism.


\begin{thebibliography}{19}

\bibitem{char65-1}
L. S. Charlap,
\emph{Compact flat Riemannian manifolds: I},
Ann. of Math. 81 (1965), 15--30. 

\bibitem{gr-ka94}
M. Grossberg and Y. Karshon,
\emph{Bott towers, complete integrability, and the extended character of 
representations}, 
Duke Math. J~76 (1994), 23--58.

\bibitem{ka-ma08}
Y. Kamishima and M. Masuda,
\emph{Cohomological rigidity of real Bott manifolds},
preprint, arXiv:0807.4263. 

\bibitem{ka-na08}
Y. Kamishima and A. Nazra,
\emph{Seifert fibered structure and rigidity on real Bott towers}, 
in preparation. 

\bibitem{masu08}
M. Masuda, 
\emph{Cohomological non-rigidity of generalized real Bott manifolds 
of height 2},
preprint, arXiv:0809.2215.

\bibitem{ma-pa08}
M. Masuda and T. Panov,
\emph{Semifree circle actions, Bott towers, and quasitoric manifolds}, 
Sbornik Math. (to appear), arXiv:math.AT/0607094. 


\bibitem{nazr08}
A. Nazra,
\emph{Real Bott tower},
Tokyo Metropolitan University, Master Thesis 2008. 



\end{thebibliography}
\end{document}